\documentclass{article}

\usepackage{amsmath}
\usepackage{amssymb}
\usepackage{amsfonts}
\usepackage{amsthm}
\usepackage{tikz}
\usepackage{color}
\usepackage{graphicx}
\usepackage{mathrsfs}
\usepackage{url}
\usepackage{setspace}
\usepackage{comment}

\newtheorem{thm}{Theorem}

\newtheorem{cnj}[thm]{Conjecture}

\newtheorem{lem}[thm]{Lemma}
\newtheorem{clm}[thm]{Claim}
\newtheorem{prb}[thm]{Problem}
\newtheorem{prp}[thm]{Proposition}
\newtheorem{cor}[thm]{Corollary}

\def\a{{\alpha}}
\def\D{{\Delta}}
\def\m{{\mu}}

\def\cF{{\cal F}}
\def\cH{{\cal H}}
\def\cI{{\cal I}}

\def\cT{{\cal T}}

\def\mt{{\emptyset}}

\def\ocFx{{\overline{\cF_x}}}

\def\dist{{\sf dist}}
\def\ekr{{\sf EKR}}

\def\hk{{\sf HK}}

\definecolor{brwn}{RGB}{140, 70, 20}
\definecolor{gren}{RGB}{  0,140, 10}

\newcommand{\gh}[1]{\textcolor{red}{\sf{#1}}}
\newcommand{\pf}[1]{\textcolor{gren}{\sf{#1}}}

\author{
P\'eter Frankl\thanks{
R\'enyi Institute, Budapest, Hungary,
\texttt{peter.frankl@gmail.com}.
}\\
Glenn Hurlbert\thanks{
Department of Mathematics and Applied Mathematics,
Virginia Commonwealth University, 
Richmond, VA, USA, 
\texttt{ghurlbert@vcu.edu}.}
}

\title{On the Holroyd-Talbot Conjecture\\ for Sparse Graphs}

\begin{document}

\maketitle

\begin{abstract}
Given a graph $G$, let $\mu(G)$ denote the size of the smallest maximal independent set of $G$.
A family of sets is called a {\it star} if some element is in every set of the family.
A {\it split} vertex has degree at least 3.
Holroyd and Talbot conjectured the following Erd\H os-Ko-Rado-type statement about intersecting families of independent sets of graphs: if $1\le r\le \mu(G)/2$ then there is an intersecting family of independent $r$-sets of maximum size that is a star.
In this paper we prove similar statements for sparse graphs on $n$ vertices: roughly, for graphs of bounded average degree with $r\le O(n^{1/3})$, for graphs of bounded degree with $r\le O(n^{1/2})$, and for trees having a bounded number of split vertices with $r\le O(n^{1/2})$.
\end{abstract}

\section{Introduction}
\label{s:Intro}

For $0\le r\le n$, let $\binom{[n]}{r}$ denote the family of $r$-element subsets ({\it $r$-sets}) of $[n]=\{1,2,\ldots,n\}$.
For any family $\cF$ of sets, define the shorthand $\cap\cF=\cap_{S\in\cF}S$.
If $\cap\cF\not=\emptyset$, we say that $\cF$ is a {\it star}; in this case, any $x\in\cap\cF$ is called a ${\it center}$.
The family $\cF_x=\{S\in\cF\ \mid\ x\in S\}$ is called the {\it full star of $\cF$ at $x$}.
Furthermore, we define the notation $\cF^r=\{S\in\cF\ \mid\ |S|=r\}$.
The family $\cF$ is {\it intersecting} if every pair of its members intersects.

Erd\H{o}s, Ko, and Rado \cite{EKR} proved the following classical theorem of central importance in extremal set theory.

\begin{thm}
\label{t:EKR}
{\bf (Erd\H os-Ko-Rado, 1961)}
If $\cF\subseteq \binom{[n]}{r}$ is intersecting for $r\leq n/2$, then $|\cF|\leq \binom{n-1}{r-1}$. 
Moreover, if $r<n/2$, equality holds if and only if $\cF=\binom{[n]}{r}_x$ for some $x\in [n]$.
\end{thm}

Hilton and Milner \cite{HilMil} proved the following stronger stability result.

\begin{thm}
\label{t:HM}
{\bf (Hilton-Milner, 1967)}
If $\cF\subseteq \binom{[n]}{r}$ is intersecting for $r\leq n/2$, and $\cF$ is not a star, then $|\cF|\leq \binom{n-1}{r-1}-\binom{n-r-1}{r-1}+1$. 
\end{thm}

For a graph $G$, let $\cI(G)$ denote the family of independent sets of $G$.
We write $s_r(v)=|\cI^r_v(G)|$ when $G$ is understood.
Let $\cF\subseteq\cI^r(G)$ be an intersecting subfamily of maximum size.
We say that $G$ is $r$-\ekr\ if some $v$ satisfies $s_r(v)=|\cF|$, and {\it strictly} $r$-\ekr\ if every such $\cF$ equals $\cI_v^r(G)$ for some $v$.

Write $\a(G)$ for the independence number of $G$.
Let $\mu(G)$ denote the size of a smallest maximal independent set of $G$. 
Equivalently, $\mu(G)$ is the size of the smallest independent dominating set of $G$.
Holroyd and Talbot \cite{HolTal} made the following conjecture.

\begin{cnj}
\label{j:HolTal}
{\bf (Holroyd-Talbot, 2005)}
For any graph $G$, if $1\leq r\leq \mu(G)/2$ then $G$ is $r$-\ekr.
\end{cnj}

Of course, this conjecture is true for the empty graph by Theorem \ref{t:EKR}.
While not explicitly stated in graph-theoretic terms, earlier results by Berge \cite{Berge}, Deza and Frankl \cite{DezFra}, and Bollob\'as and Leader \cite{BolLead} support the conjecture.
For example, the case of $G$ equal to a disjoint union of $k$ complete graphs of sizes $n_1\le\cdots\le n_r$ was verified (in fact for all $r\le\a(G)$) in \cite{BolLead,DezFra} for the uniform case $2\le n_1=\cdots=n_k$, in \cite{HolSpeTal} for the non-uniform case $2\le n_1\le\cdots\le n_k$, and in \cite{Bey} for the general case.
The cases of $G$ being a power of either a path \cite{HolSpeTal} or a cycle \cite{Talbot}, or a {\it special chain} (essentially, a path of complete graphs of increasing size) or the disjoint union of two special chains \cite{HurKamChord}, were both verified for all $r\le\a(G)$ as well.
The conjecture has been proven for $\mu(G)$ sufficiently large in terms of $r$ \cite{Borg}, and also for various graph classes, for example, disjoint unions of complete graphs, paths, and cycles containing at least one isolated vertex \cite{BorHol,HolSpeTal}, disjoint unions of complete multipartite graphs containing at least one isolated vertex \cite{BorHolDouble}, disjoint unions of length-2 paths \cite{FegHurKam}, chordal graphs containing an isolated vertex \cite{HurKamChord}, and others.
In fact, for the cases of complete graphs and cycles just mentioned, \cite{BorHol} extends the range of $r$ beyond $\m(G)/2$ to $\a(G)/2$.
One can observe, for example, that the complete $k$-partite graph $G=K_{n_1,\ldots,n_k}$ is $r$-\ekr\ for all $r\le\a(G)/2$, because every independent set is contained in some part.
However, $G$ is not $r$-\ekr\ for $\a(G)/2 < r \le \a(G)$.

For vertices $u$ and $v$ in a graph $G$, we use the notations $\deg_G(u)$ and $\dist_G(u,v)$ for the degree of vertex $u$ and the distance between $u$ and $v$ in $G$, respectively; we may omit the subscript if the context is clear.


\section{Results}

Here we prove the following theorem.

\begin{thm}
\label{t:HolTalCnj1}
Let $r$ and $d$ be positive integers.
Suppose that $G$ is a graph on $n>\frac{27}{8}dr^2$ vertices, having maximum degree less than $d$.
Then $G$ is $r$-\ekr.
\end{thm}

We can expand the class of graphs beyond bounded degree to bounded average degree at the cost of reducing the range of $r$ from $O(n^{1/2})$ to $O(n^{1/3})$, as follows.

\begin{thm}
\label{t:HolTalCnj2}
Given a positive integer $r$, let $c\ge e/36$ be a constant.
Suppose that $G$ is a graph on $n>18cr^3$ vertices, having at most $cn$ edges.
Then $G$ is $r$-\ekr.
\end{thm}

It is likely that a quadratic bound on $n$ is possible for Theorem \ref{t:HolTalCnj2} as well.
Note that the case $c=1$ in Theorem \ref{t:HolTalCnj2} is especially relevant for trees.
In this case, we can retrieve a quadratic lower bound for $n$ for one special class of trees.

A {\it split vertex} in a graph is a vertex of degree at least three.
A {\it spider} is a tree with exactly one split vertex.
For a spider $S$ with split vertex $w$ and leaves $v_1,\ldots,v_k$, we write $S=S(\ell_1,\ldots,\ell_k)$, where $\ell_i=\dist(w,v_i)$.
The notation is written in {\it spider order} when the following conditions hold:
\begin{itemize}
    \item 
    if $\ell_i$ and $\ell_j$ are both odd and $\ell_i<\ell_j$ then $i<j$;
    \item
    if $\ell_i$ and $\ell_j$ are both even and $\ell_i<\ell_j$ then $i>j$; and
    \item
    if $\ell_i$ is odd and $\ell_j$ is even then $i<j$.
\end{itemize}
Notice that, since every independent set of $S(1,1,\ldots,1)$ is a subset of its leaves, Conjecture \ref{j:HolTal} is true for $S(1,1,\ldots,1)$.
In an attempt to prove the Holroyd-Talbot conjecture for spiders by induction, the authors of \cite{HurKamTrees} proved the following result.

\begin{thm}
\label{t:SpiderOrder}
{\bf (Hurlbert-Kamat, 2022)}
Suppose that $S=S(\ell_1,\ldots,\ell_k)$ is a spider written in spider order.
Let $w$ be the split vertex of $S$, for each $i$ let $u_i$ be any vertex on the $wv_i$-path, and suppose that $r\le\a(S)$.
Then
\begin{enumerate}
    \item 
    \label{p:root}
    $s_r(w)\le s_r(v_i)$ for all $i$,
    \item 
    \label{p:legs}
    $s_r(u_i)\le s_r(v_i)$ for all $i$, and
    \item 
    \label{p:leaves}
    $s_r(v_j)\le s_r(v_i)$ for all $i<j$.
\end{enumerate}
\end{thm}

Estrugo and Pastine \cite{EstrPast} call a tree $T$ $r$-\hk\ if $s_r(v)$ is maximized at a leaf of $T$ (and \hk\ if $r$-\hk\ for all $r\le\a(T)$).
It is proved in \cite{HurKamChord} that every tree is $r$-\hk\ for $r\leq 4$, but Baber \cite{Babe}, Borg \cite{BorCounter}, and Feghali, Johnson, and Thomas \cite{FegJohnTho} each found counterexamples when $r\ge 5$.
However, parts \ref{p:root} and \ref{p:legs} of Theorem \ref{t:SpiderOrder} together imply that every spider $S$ is \hk.
Theorem \ref{t:HolTalCnj2} shows that spiders are $r$-\ekr\ for $r<(n/18)^{1/3}$.
Unfortunately, $\mu/2$ for spiders is roughly $n/6$, so there remains a big gap.
Our next theorem shrinks that gap somewhat.

\begin{thm}
\label{t:SpiderEKR}
Let $S=S(\ell_1,\ldots,\ell_k)$ be a spider on $n$ vertices, with split vertex $w$ and leaves $v_1,\ldots,v_k$.
Suppose that $r\le\sqrt{n\ln 2}-(\ln 2)/2$.
Then $S$ is $r$-\ekr.
\end{thm}

We note that every spider $S$ has $\a(S)= 1> \sqrt{n\ln 2}-(\ln 2)/2$ for $n\le 2$, $\a(S)= 2> \sqrt{n\ln 2}-(\ln 2)/2$ for $n=3$, and $\a(S)\ge (n-1)/2> \sqrt{n\ln 2}-(\ln 2)/2$ for $n\ge 4$.
In other words, the hypothesis of Theorem \ref{t:SpiderEKR} implies that $r\le\a(S)$ for all $n$.

Finally, we prove the following similar result for more general trees.

\begin{thm}
\label{t:SplitEKR}
Let $T$ be a tree on $n$ vertices, with exactly $s>1$ split vertices.
Suppose that $1<s<r/2$ and $r\le\sqrt{n\ln c}-(\ln c)/2$, where $c=2-2s/r$.
Then $T$ is $r$-\ekr.
\end{thm}


\section{Technical Lemmas}

\begin{prp}
\label{p:Estim}
If $0\le x\le 2k/(k+1)^2$ for some $k\ge 1$, then $e^{-x}<1-\left(\frac{k}{k+1}\right)x$.
\end{prp}

\begin{proof}
Let $0\le x\le 2k/(k+1)^2$ for some $k\ge 1$.
Then $|x|<1$, and so $e^{-x}=\sum_{i\ge 0}(-x)^i/i! < 1-x+x^2/2$.
Also, $(k+1)x<2$, which implies that $x^2/2 < x/(k+1) = [1-k/(k+1)]x$.
Thus $e^{-x} < 1-x+x^2/2 < 1-\left(\frac{k}{k+1}\right)x$.
\end{proof}

\begin{cor}
\label{c:Estim}
If $0\le y\le 2k^2/(k+1)^3$ for some $k\ge 1$, then $1-y>e^{-\left(\frac{k+1}{k}\right)y}$.
\end{cor}

\begin{proof}
Set $x=\left(\frac{k+1}{k}\right)y$ and apply Proposition \ref{p:Estim}.
\end{proof}

\begin{lem}
\label{l:estim}
If $r\ge 2$, $d\ge 2$, and $n\ge \frac{27}{8}dr^2$, then $\prod_{i=1}^{r-1}\left(1-\frac{r+id}{n}\right) > \frac{r}{n}$.
\end{lem}

\begin{proof}
We begin with
\[\prod_{i=1}^{r-1}\left(1-\frac{r+id}{n}\right)
\ge 1-\sum_{i=1}^{r-1}\frac{r+id}{n}
= 1-\frac{r(r-1)+d\binom{r}{2}}{n}
= 1-\frac{(d+2)\binom{r}{2}}{n}.\]
Since $d\ge 2$, and by using Corollary \ref{c:Estim} with $y=dr^2/n$ and $k=2$, we have
\[1-\frac{(d+2)\binom{r}{2}}{n}
> 1-\frac{dr^2}{n}
> e^{-3dr^2/2n}
> e^{-4/9}
> .64\ .\]
In addition, we calculate
\[\frac{r}{n}
\le \frac{8}{27dr}
\le \frac{2}{27}
< .08\ ,\]
which completes the proof.
\end{proof}


\begin{clm}
\label{cm:star}
Let $G$ be a graph with $n$ vertices and maximum degree less than $d$.
Then every vertex $v$ satisfies
\[s_r(v)\ge \frac{1}{(r-1)!}(n-d)(n-2d)\cdots (n-(r-1)d).\]
\end{clm}

\begin{proof}
Let $W_0$ be the set of vertices of $G$, and set $w_0=v$.
For each $0<i<r$, choose $w_i\in W_i$, where $W_{i+1}=W_i-N[w_i]$.
Then by induction we have $|W_i|\ge n-id$ for each such $i$.
The resulting set $\{w_0,\ldots,w_{r-1}\}$ is independent in $G$ and there are at least $\prod_{0<i<r}(n-id)$ ways to choose such sets, ignoring replication.
Accounting for replication, we obtain the result.
\end{proof}

\begin{lem}
\label{l:BigStar}
Let $H$ be a graph with at least $m=n(1-1/3r)$ vertices and maximum degree less than $d$.
Suppose that $1/3r+rd/n\le 2k^2/(k+1)^3$ for some $k\ge 1$.
Then every vertex $v$ satisfies \[s_r(v) \ge \frac{n^{r-1}}{(r-1)!}e^{-(r-1)2k/(k+1)^2}.\]
\end{lem}

\begin{proof}
We use Claim \ref{cm:star} and Corollary \ref{c:Estim} with $y=1/3r+rd/n$ to obtain
\begin{align*}
    s_r(v)
    &\ge \frac{1}{(r-1)!}\prod_{0<i<r} (m-id)
        \ge \frac{n^{r-1}}{(r-1)!}\prod_{0<i<r} \left(1-\frac{1}{3r}-\frac{id}{n}\right)\\
    &\ge \frac{n^{r-1}}{(r-1)!}\prod_{0<i<r} \left[1-\left(\frac{1}{3r}+\frac{rd}{n}\right)\right]
        \ge \frac{n^{r-1}}{(r-1)!}\prod_{0<i<r} e^{-\left(\frac{k+1}{k}\right)\left(\frac{1}{3r}+\frac{rd}{n}\right)}\\
    &\ge \frac{n^{r-1}}{(r-1)!}e^{-(r-1)\left(\frac{k+1}{k}\right)\left(\frac{1}{3r}+\frac{rd}{n}\right)}
        \ge \frac{n^{r-1}}{(r-1)!}e^{-(r-1)2k/(k+1)^2}.
\end{align*}
\end{proof}


\section{Proof of Theorem \ref{t:HolTalCnj1}}
\label{s:New}

We use the following result of Frankl \cite{Frankl}.
For $\cF\subseteq\binom{[n]}{r}$, define $\ocFx=\cF-\cF_x$.

\begin{thm}
\label{t:Frankl}
{\bf (Frankl, 2020)}
If $\cF\subseteq \binom{[n]}{r}$ is intersecting and $r<n/72$,
then there is some $x$ such that $|\ocFx|\leq \binom{n-3}{r-2}$. 
\end{thm}


\noindent
{\it Proof of Theorem \ref{t:HolTalCnj1}.}
The result is trivial for $r=1$ or $d=1$, so we assume $r\ge 2$ and $d\ge 2$.
Let $x$ be as in Theorem \ref{t:Frankl}, and select $E\in\ocFx$, which we may assume to be nonempty.
Via the same counting method as in Claim \ref{cm:star}, we have at least
\begin{equation}
\label{e:xnotE}
\frac{1}{(r-1)!}(n-r-d)(n-r-2d)\cdots (n-r-(r-1)d)
\end{equation}
$r$-sets $F\in\cI_r(x)$ with $F\cap E=\mt$.
Since $\cF$ is intersecting, these sets are not in $\cF_x$.
Therefore, using Theorem \ref{t:Frankl} and the bound in (\ref{e:xnotE}), we have
\begin{align*}
    |\cF|
    &= |\cF_x|+|\ocFx|\\
    &\le |\cI_r(x)| - \frac{(n-r-d)\cdots (n-r-(r-1)d)}{(r-1)!} +\binom{n-3}{r-2}.
\end{align*}
This upper bound is at most $|\cI_r(x)|$ precisely when
\[\binom{n-3}{r-2}\le \frac{1}{(r-1)!}\prod_{i=1}^{r-1}(n-r-id),\]
which we rewrite as
\[\prod_{i=1}^{r-1}(n-r-id)\ge (r-1)!\binom{n-3}{r-2} = (r-1)\prod_{i=1}^{r-2}(n-2-i).\]
This inequality will follow from showing that
\[\prod_{i=1}^{r-1}(n-r-id)\ge rn^{r-2},\]
which holds by Lemma \ref{l:estim}, and which completes the proof.


\section{Proof of Theorem \ref{t:HolTalCnj2}}
\label{s:Old}

The result is trivial for $r=1$, so we may assume that $r\ge 2$.
Let $V_0$ be the set of vertices of $G$.
For each $i\ge 0$, choose $v_i\in V(G_i)$ such that $\deg_{G_i}(v_i)\ge 3cr$, where $G_{i+1}=G_i-v_i$.
Let $t$ be minimum such that $\D(G_t)<3cr$.
The number of edges removed in this process is at least $3tcr$, which must be at most the number of edges of $G$; thus $t\le n/3r$.
Hence $V(G_t)=n-t\ge n(1-1/3r)$.

Now we set $d=3cr$, $k=4r-7\ge 1$, and calculate that
\[
(k+3)+\left(\frac{3k+1}{k^2}\right)
\le k+7
=4r,
\]
so that $(k+1)^3\le 4k^2r$, which implies that
\[
\frac{1}{3r}+\frac{rd}{n}
< \frac{1}{3r}+\frac{3cr^2}{18cr^3}\\
= \frac{1}{2r}\\
\le \frac{2k^2}{(k+1)^3}.\\
\]
This allows the use of Lemma \ref{l:BigStar} with $H=G_t$, $m=n(1-1/3r)$, and $d=3cr$.
We obtain that each vertex $v$ of $G_t$ has $s_r(v)$ at least
\begin{equation}
\label{e:BigStar}
    \frac{n^{r-1}}{(r-1)!}e^{-(r-1)2k/(k+1)^2}.
\end{equation}

Now we use Theorem \ref{t:HM} to show that any intersecting family $\cF$ of independent $r$-sets that is not a star has size less than (\ref{e:BigStar}).
First, we note the combinatorial identity $\binom{n-1}{r-1}-\binom{n-r-1}{r-1}+1 = 1 + \binom{n-2}{r-2} + \binom{n-3}{r-2} + \cdots + \binom{n-r-1}{r-2}$.
Second, we observe the inequality $r^2/n< e^{-(r-1)2k/(k+1)^2}$.
Indeed, 
\[
\frac{r^2}{n}
< \frac{1}{18cr}
\le e^{-1} 
\le e^{-(r-1)(8r-14)/(4r-6)^2}
= e^{-(r-1)2k/(k+1)^2},
\]
because $c\ge e/36$ and $(4r-6)^2 > (r-1)(8r-14)$ (since $r\ge 2$).

Finally, if $\cF$ is as above, then we have
\[
|\cF|
< r\binom{n-2}{r-2}
= \frac{r(r-1)}{n-1}\binom{n-1}{r-1}
< \frac{r^2}{n}\cdot\frac{n^{r-1}}{(r-1)!}
< \frac{n^{r-1}}{(r-1)!}e^{-(r-1)2k/(k+1)^2}.
\]
This finishes the proof.


\section{Proof of Theorem \ref{t:SpiderEKR}}

\begin{lem}
\label{l:SpiderStar}
Let $S=S(\ell_1,\ldots,\ell_k)$ be a spider on $n$ vertices and let $v$ be a leaf of $S$.
Suppose that $r\le\a(S)$.
Then
\[s_r(v)\ge \binom{n-r-1}{r-1}+\binom{n-k-r-2}{r-2}.\]
\end{lem}

\begin{proof}
Let $S=S(\ell_1,\ldots,\ell_k)$, in spider order.
We may assume that $v=v_k$ and then use Theorem \ref{t:SpiderOrder} for the other leaves.
For $S(1,1,\ldots,1)$ we have $s_r(v)=\binom{n-2}{r-1}$ and $k=n-1$, so that $\binom{n-k-r-2}{r-2}=0$ and $\binom{n-2}{r-1} \ge \binom{n-r-1}{r-1}$.
Thus we may assume that $\ell_k\ge 2$, implying that $v$ and $w$ are not adjacent.

We first count the number of independent $r$-sets containing $v$ that do not contain the split vertex $w$.
The number of such sets is
\[|\cI_v^r(S-w)| = |\cI_v^r(\cup_{i=1}^kP_{\ell_i})|,\]
where $P_{\ell_i}$ denotes the path on $\ell_i$ vertices.

Next we add edges to the disjoint union of paths, joining the many paths together to form one long path, thus reducing the number of independent $r$-sets that contain $v$ but not $w$.
For each $1\le i\le k$, let $u_i$ be the neighbor of $w$ on the $wv_i$-path in $S$; that is, the endpoint of the $i^{\rm th}$ path of $S-w$ that is different from $v_i$.
Now, for each $1\le i<k$, add the edge $v_iu_{i+1}$.
Finally, remove $v$ and its unique neighbor, resulting in the graph $P_m$, for $m=n-3$.
This results in the inequality
\[|\cI_v^r(\cup_{i=1}^kP_{\ell_i})|\ge |\cI^{r-1}(P_m)|.\]

We relabel the vertices of $P_m$ as $x_1,\ldots,x_m$, in order.
Observe that $\{x_{a_1},$ $x_{a_1+a_2},$ $\ldots,$ $x_{a_1+\cdots+a_{r-1}}\}$ is independent in $P_m$ if and only if
\begin{equation}
\label{e:sum1}
    \sum_{i=1}^ra_i=m,\ a_1\ge 1,\ a_i\ge 2\ \mbox{for}\ 1<i<r,\ \mbox{and}\ a_r=m-a_{r-1}\ge 0.
\end{equation}
Set $b_1=a_1-1$, $b_i=a_i-2$ for $1<i<r$, and $b_r=a_r$.
Then system (\ref{e:sum1}) can be rewritten as
\begin{equation}
\label{e:sum2}
    \sum_{i=1}^rb_i=m-2r+3=n-2r,\ \mbox{with}\ b_i\ge 0,\ \mbox{for all}\ 1\le i\le r.
\end{equation}
It is well known that the number of integer solutions to system (\ref{e:sum2}) equals
\[\binom{n-2r+r-1}{r-1}=\binom{n-r-1}{r-1}.\]

Second, we count the number of independent $r$-sets containing $v$ that also contain the split vertex $w$.
The number of such sets equals
\[|\cI_v^{r-1}(S-N[w])| = |\cI_v^{r-1}(\cup_{i=1}^kP_{\ell_i-1})|.\]

As above, we add edges to the disjoint union of paths, to reduce the number of independent $r$-sets that contain $v$ and $w$.
For each $1\le i\le k$, let $u'_i$ be the neighbor of $u_i$ other than $w$ on the $wv_i$-path in $S$.
Now, for each $1\le i<k$, add the edge $v_iu'_{i+1}$.
Finally, remove $v$ and its unique neighbor, resulting in the graph $P_{m'}$, for $m'=n-3-k$.
This results in the inequality
\[|\cI_v^{r-1}(\cup_{i=1}^kP_{\ell_i-1})|\ge |\cI^{r-2}(P_{m'})|.\]

Counting via the same method as above, we obtain
\[|\cI^{r-2}(P_{m'})| = \binom{n-k-r-2}{r-2}\]
such sets, which completes the proof.
\end{proof}


\noindent
{\it Proof of Theorem \ref{t:SpiderEKR}}.
It is easy to check that $r\le\sqrt{n\ln 2}-(\ln 2)/2$ implies that $r^2\le (n-r)\ln 2$.
We use this in the calculations below.

Using Lemma \ref{l:SpiderStar} with Theorem \ref{t:HM}, as in the proof of Theorem \ref{t:HolTalCnj2}, the result will follow from proving the inequality
\begin{equation}
\label{e:binoms}
    \binom{n-1}{r-1} < 2\binom{n-r-1}{r-1}.
\end{equation}
To accomplish this, we denote $m^{\underline{t}}=m!/(m-t)!$ and calculate the ratio
\begin{align}
    \binom{n-1}{r-1}\Big/\binom{n-r-1}{r-1}
    &= \frac{(n-1)^{\underline{r-1}}}{(n-r-1)^{\underline{r-1}}}
        \le \frac{(n-r+1)^{r-1}}{(n-2r+1)^{r-1}} \nonumber\\
    &= \left(\frac{n-2r+1}{n-r+1}\right)^{-(r-1)}
        = \left(1-\frac{r}{n-r+1}\right)^{-(r-1)} \nonumber\\
    &\le e^{r(r-1)/(n-r+1)}
        < e^{r^2/(n-r)} \label{i:IneqSquare}\\
    &\le e^{\ln 2}
        = 2, \nonumber
\end{align}
which finishes the proof.


\section{Proof of Theorem \ref{t:SplitEKR}}

\begin{lem}
\label{l:SplitStar}
Let $T$ be a tree on $n$ vertices with exactly $s>1$ split vertices, and let $v$ be a leaf of $T$.
Suppose that $r\le\a(T)$.
Then
\[s_r(v)\ge \binom{n-r-s}{r-1}+1.\]
\end{lem}

\begin{proof}
Let $W$ denote the set of split vertices of $T$.
We need only count the number of independent $r$-sets containing $v$ that do not contain any split vertex.
The number of such sets equals
\[
\cI_v^r(T-W)| 
> |\cI_v^r(P_{n-s})|
= |\cI^{r-1}(P_{n-s-2})|
= \binom{n-r-s}{r-1},
\]
as in the proof of Lemma \ref{l:SpiderStar}.

The strict inequality comes from the existence of at least one independent $r$-set of $T-W$ that is not independent in $P_{n-s}$ because of the joining of the many paths that create $P_{n-s}$.
For example, let $P'$ and $P''$ be two paths in $T-W$ that are consecutive in $P_{n-s}$, with endpoints $u'\in P'$ and $u''\in P''$ such that $u'$ is adjacent to $u''$ in $P_{n-s}$.
Let $A\in\cI_{v}^{r}(P_{n-s})$, define $a'$ to be the vertex in $A$ that is closest to $u'$, $a''$ to be the vertex in $A-\{a'\}$ that is closest to $u''$, and $A'=(A-\{a',a''\})\cup\{u',u''\}$.
Then $A'\in\cI_{v}^{r}(T-W)-\cI_{v}^{r}(P_{n-s})$.
\end{proof}


\noindent
{\it Proof of Theorem \ref{t:SplitEKR}.}
As in the proof of Theorem \ref{t:SpiderEKR}, we use Lemma \ref{l:SplitStar} and Theorem \ref{t:HM}, which reduces the proof to certifying the inequality
\begin{equation}
\label{e:binoms2}
    \binom{n-1}{r-1}\le \binom{n-r-1}{r-1} + \binom{n-r-s}{r-1}.
\end{equation}
Suppose that $1<s<r/2$ and $r\le \sqrt{n\ln c}-(\ln c)/2$, where $c=2-2s/r$.
Let $a=\frac{r^2}{\ln c}+r$ and $b=\frac{(r+2)^3}{2(r+1)}+r+s-1$.
By rearranging the given condition on $r$, we obtain $n\ge a+\frac{\ln c}{4} > a$.
Now let $d=2(r+1)\ln c$ so that we have
\begin{align*}
d(a-b)
&= 2(r+1)r^2 - (\ln c)\left[(r+2)^3 + 2(r+1)(s-1)\right]\\
&> 2(r+1)r^2 - (\ln 2)\left[(r+2)^3 + (r+1)(2s-2)\right]\\
&> 2(r+1)r^2 - 0.7\left[(r+2)^3 + (r+1)(r-2)\right]\\
&= 2r^3+2r^2 - 0.7\left(r^3+7r^2+11r+6\right)\\
&= 1.3r^3 - 2.9r^2 - 7.7r - 4.2\\
&> 0
\end{align*}
since $r\ge 5$.
Because $a-b>0$ and $n>a$, we have $n>b$, which is equivalent to
\begin{equation}
    \label{i:estimbound}
    \frac{r+1}{n-r-s+1} < \frac{2(r+1)^2}{(r+2)^3}.
\end{equation}
Next, we derive the following estimates, using Inequality \ref{i:estimbound} to access Corollary \ref{c:Estim} with $y=(r+1)/(n-r-s+1)$ and $k=r+1$.
\begin{align*}
    \frac{\binom{n-r-1}{r-1} + \binom{n-r-s}{r-1}}{\binom{n-r-1}{r-1}}
    &= 1 + \frac{(n-2r)^{\underline{s-1}}}{(n-r-1)^{\underline{s-1}}}
        \ge 1 + \left(\frac{n-2r-s+2}{n-r-1-s+2}\right)^{s-1}\\
    &> 1 + \left(\frac{n-2r-s}{n-r-s+1}\right)^s
        = 1 + \left(1 - \frac{r+1}{n-r-s+1}\right)^s\\
    &> 1 + e^{-\left(\frac{r+2}{r+1}\right)\left(\frac{r+1}{n-r-s+1}\right)s}
        > 1 + e^{-\left(\frac{r+2}{r+1}\right)\big(\frac{2(r+1)^2}{(r+2)^3}\big)s}\\
    &= 1 + e^{-\big(\frac{2(r+1)}{(r+2)^2}\big)s}
        > 1 + e^{-2s/r}
        > 2 - 2s/r.
\end{align*}
The assumption that $s<r/2$ makes the final result greater than $1$.
Finally, we follow Inequality (\ref{i:IneqSquare}), since $r\le\sqrt{n\ln c}-(\ln c)/2$ implies that $r\le\sqrt{n\ln 2}-(\ln 2)/2$, and calculate the ratio
\[\binom{n-1}{r-1}\Big/\binom{n-r-1}{r-1}
< e^{r^2/(n-r)}
\le e^{\ln (2 - 2s/r)}
= 2 - 2s/r,
\]
which finishes the proof.


\section{Questions and Remarks}

It is clear that improving the orders of magnitude in the upper bound on $r$ in our results will require techniques other than comparison to the Hilton-Milner bounds.
To that end, the specificity of spider structure and the knowledge of the location of their biggest stars begs for a proof that they are $r$-\ekr\ for $r\le\mu/2$ (or possibly $r\le\a$).

Along these lines, consider the family $\cT$ of all trees having no vertex of degree 2.
The authors of \cite{HurKamTrees} conjecture that every tree in $\cT$ is \hk.
Naturally, we believe that such trees are $r$-\ekr\ for all $r\le\m(T)$ as well.
As a first step in this direction, for $i\in\{1,2,3\}$, let $T_i(h)$ be a complete binary tree of depth $h$ (i.e. having $2^{h+1}-1$ vertices), with root vertex $v_i$.
Note that $v_i$ is the unique degree-$2$ vertex in $T_i(h)$.
Now define the tree $T(h)$ by $V(T(h))=\{w\}\cup_{i=1}^3V(T_i(h))$, with $w$ adjacent to each $v_i$.
Then $T(h)\in\cT$.

\begin{prb}
\label{p:tstarEKR}
Show that $T(h)$ is $r$-\ekr\ for all $r\le\mu(T(h))/2$.
\end{prb}

Finally, we say that a family $\cF$ of sets is \ekr\ if it has the property that if $\cH$ is an intersecting subfamily of $\cF$ then there is some element $x$ such that $|\cH|\le |\cF_x|$, and that a graph $G$ is \ekr\ if $\cI(G)$ is \ekr.
We observe that the non-uniform case --- considering $\cI(G)$ instead of $\cI^r(G)$ --- has yet to be studied specifically for graphs.
Of course, this is a special case of Chv\'atal's conjecture (see \cite{Chvatal}) that every subset-closed family $\cF$ of sets is \ekr.
For example, by a result of \cite{Snevily}, every graph with an isolated vertex is \ekr.
Also, powers of paths or cycles (resp. special chains) are \ekr\ by the results of \cite{HolSpeTal,Talbot} (resp. \cite{HurKamChord}) for fixed $r$ because we can use the same star center for each $r$.
The same holds for disjoint unions of complete graphs because the star center is either an isolated vertex, if it exists, or a vertex in a smallest component.
Any vertex-transitive graph $G$ that is $r$-\ekr\ for all $r\le\a(G)$ would also be \ekr.
It is conjectured in \cite{FegHurKam} that if $G$ is a disjoint union of length-2 paths then it is $r$-\ekr\ for all $\mu(G)/2<r\le\a(G)/2$.
It may also be true for all $r\le\a(G)$, which would imply that $G$ is \ekr\ because the largest star is always centered on a leaf, and all leaves look alike.
Additionally, if one could prove that every spider $S$ is $r$-\ekr\ for all $r\le\a(S)$ then it would follow from Theorem \ref{t:SpiderOrder} that spiders are \ekr.
Of course, while complete $k$-partite graphs $G$ are not $r$-\ekr\ for $\a(G)/2<r\le\a(G)$, that does not mean that they are not \ekr.

\section*{Acknowledgement}

We thank the referees for making several important comments that improved the exposition and corrected some of the calculations.


\end{document}